\documentclass[11pt]{article}
\pdfoutput=1
\usepackage{url}
\usepackage{amsmath}
\usepackage{amssymb}
\usepackage{amsfonts}
\usepackage{euscript}
\usepackage{theorem}
\usepackage[all]{xypic}
\usepackage[utf8]{inputenc}
\usepackage{bbding} % for ticks and crosses in provability diagram

\usepackage{tikz}
\usetikzlibrary{arrows,matrix}

\newtheorem{theorem}{Theorem}[section]
\newtheorem{proposition}[theorem]{Proposition}
\newtheorem{lemma}[theorem]{Lemma}
\newtheorem{corollary}[theorem]{Corollary}

{\theorembodyfont{\rmfamily}
 \newtheorem{definition}[theorem]{Definition}
}

\newenvironment{proof}{\medskip\noindent\emph{Proof.}}{\hfill$\Box$\medskip}

\newenvironment{proofof}[1]{\medskip\noindent\emph{Proof of #1.}}{\hfill$\Box$\medskip}

\newcommand{\into}{\hookrightarrow}

\newcommand{\epito}{\twoheadrightarrow}
\newcommand{\pow}[1]{\mathcal{P}(#1)}
\newcommand{\iso}{\cong}

\newcommand{\id}[1]{\mathrm{id}_{#1}}

\newcommand{\set}[1]{\{#1\}}
\newcommand{\such}{\mid}
\newcommand{\lthen}{\Rightarrow}
\newcommand{\liff}{\Leftrightarrow}

\newcommand{\la}{\langle}
\newcommand{\ra}{\rangle}

\newcommand{\tbigwedge}{\bigvee\nolimits}
\newcommand{\tbigvee}{\bigwedge\nolimits}
\newcommand{\tsup}{\sup\nolimits}

\newcommand{\all}[3]{\forall\, #1 \,{\in}\, #2\,.\left(#3\right)}

\newcommand{\xall}[3]{\forall\, #1 \,{\in}\, #2\,.\,#3}
\newcommand{\xsome}[3]{\exists\, #1 \,{\in}\, #2\,.\,#3}

\newcommand{\xuall}[2]{\forall\, #1\,.\,#2}

\newcommand{\Ch}[1]{\mathrm{Ch}({#1})}

\newcommand{\Prog}[1]{\mathrm{Prog}(#1)}

\newcommand{\tuple}[2]{\langle #1 \rangle}

\newcommand{\up}[1]{{\uparrow}#1}

\newcommand{\cat}[1]{\mathcal{#1}}

\newcommand{\arr}{{\cdot \rightarrow \cdot}}
\newcommand{\E}{\cat{E}}
\newcommand{\EBW}{\cat{E}_\BW}
\newcommand{\EBWPf}{{\EBW [P,f]}}
\newcommand{\Efree}{\cat{E}_\textit{free}}
\newcommand{\Eff}{\mathsf{Eff}}
\newcommand{\F}{\cat{F}}
\newcommand{\Set}{\mathsf{Set}}
\newcommand{\Sets}{\Set}
\newcommand{\W}{\mathcal{W}}
\newcommand{\RT}[1]{\mathsf{RT}(#1)}

\newcommand{\BW}{\textit{BW}}
\newcommand{\ev}{\operatorname{ev}}
\newcommand{\im}{\operatorname{im}}
\newcommand{\rk}{\operatorname{rk}}

\newcommand{\N}{\mathbb{N}}
\newcommand{\On}{\mathcal{O}} 
\newcommand{\Onwk}{\mathcal{O}'} 
\newcommand{\LL}{\mathcal{L}}

\newcommand{\restrict}[1]{|_{#1}}

\newcommand{\univ}{\textit{univ}}

\begin{document}

\title{On the Bourbaki-Witt Principle in Toposes}

\author{
Andrej Bauer\\
University of Ljubljana, Slovenia\\
\texttt{Andrej.Bauer@andrej.com}
\and
Peter LeFanu Lumsdaine\\
Dalhousie University, Halifax, Canada\\
\texttt{p.l.lumsdaine@mathstat.dal.ca}
}

\maketitle

\begin{abstract}
The Bourbaki-Witt principle states that any progressive map on a
chain-complete poset has a fixed point above every point.  It is
provable classically, but not intuitionistically.

We study this and related principles in an intuitionistic setting.
Among other things, we show that Bourbaki-Witt fails exactly when the
trichotomous ordinals form a set, but does not imply that fixed points
can always be found by transfinite iteration.  Meanwhile, on the side
of models, we see that the principle fails in realisability toposes,
and does not hold in the free topos, but does hold in all cocomplete
toposes.
\end{abstract}

\section{Introduction}
\label{sec:introduction}

The Bourbaki-Witt
theorem~\cite{bourbaki49:_sur_le_theor_de_zorn,witt51:_beweis_zum_satz_von_m}
states that a progressive map $f : P \to P$ on a chain-complete
poset~$P$ has a fixed point above every point. (A map is
\emph{progressive} if $x \leq f(x)$ for all $x \in P$.)
A classical proof of the Bourbaki-Witt theorem constructs the
increasing sequence
\begin{equation*}
  x \leq f(x) \leq f^2(x) \leq \cdots \leq f^\omega(x) \leq
  f^{\omega+1}(x) \leq \cdots
\end{equation*}
where chain-completeness is used at limit stages. If the sequence
is indexed by a large enough ordinal, it must stabilise, giving a fixed point
of~$f$ above~$x$.

It has been observed recently by the first author~\cite{Bauer:On_the_Failure} that in the
effective topos there is a counterexample to the Bourbaki-Witt theorem, as well as to the
related Knaster-Tarski theorem. An earlier result of Rosolini~\cite{RosoliniG:modpti}
exhibits a model of intuitionistic set theory in which the (trichotomous) ordinals form a
set, and since the successor operation has no fixed points, this also provides a
counterexample to intuitionistic validity of the Bourbaki-Witt theorem.

The counterexamples bury any hope for an intuitionistic proof of the Bourbaki-Witt
theorem. However, several questions still remain. Is the theorem valid in other toposes?
How is it linked with the existence of large enough ordinals? How does it compare to
Knaster-Tarski and other related fixed-point principles? We address these questions in the
present paper.

\subsection{Overview}
\label{subsec:overview}

After laying out the setting in Section~\ref{sec:preliminaries}, we begin in
Subsection~\ref{subsec:related-fixed-point} by summarising the relationships between
various fixed-point principles of the same form as the Bourbaki-Witt principle. In
Subsection~\ref{subsec:subtleties}, we discuss several classically equivalent formulations
of the Bourbaki-Witt principle, which turn out to be intuitionistically equivalent as
well. Likewise, several ways of stating that the Bourbaki-Witt theorem fails are
intuitionistically equivalent. In Subsection~\ref{subsec:ordinals}, we investigate the
connection between the Bourbaki-Witt principle and iteration along ordinals, and prove that
failure of the principle is equivalent to the trichotomous ordinals forming a set.

In Section~\ref{sec:models}, we change tack and investigate validity of the Bourbaki-Witt
principle in various toposes. First we show that realisability toposes contain
counterexamples to the principle. From this we conclude that the principle cannot hold in
the free topos, as there is a definable chain-complete poset with a definable progressive
map which is interpreted as a counterexample in the effective topos. Next we show that the
Bourbaki-Witt principle transfers along geometric morphisms, and hence its validity in the
category of classical sets implies validity in cocomplete toposes, so in particular in
Grothendieck toposes. Finally, we show by topos-theoretic means that while the Bourbaki-Witt
principle does imply that the ordinals cannot form a set, it does not imply that
fixed-points can always be found by iteration along ordinals, as they can classically.

%%% Local Variables: 
%%% mode: latex
%%% TeX-master: "bw"
%%% End: 

\section{Preliminaries}
\label{sec:preliminaries}

% I removed the subsections because the whole section is very short and because
% they created an unfortunate page break which split the statement of Bourbaki-Witt
% across two pages. (Andrej)

%\subsection{Logical setting}
%\label{subsec:logical-setting}

The content of this paper takes place in two different logical settings. In the first
setting we put on our
constructive hats and prove theorems in intutionistic mathematics. Our proofs are written
informally but rigorously in the style of Errett Bishop (but without countable choice).
They can be interpreted in any elementary topos with natural numbers
object~\cite{lambek-and-scott, maclane92:_sheav_geomet_logic}, or in an intuitionistic set
theory such as IZF~\cite{Aczel-Rathjen:Notes_on_CST}. Since unbounded quantification is
not available in topos logic, statements referring to all structures of a certain kind are to
be interpreted as schemata, as is usual in that setting. When we meet a statement with an inner
unbounded quantifer, we discuss it explicitly. Intuitionistic set theories do not suffer
from this complication.

In the second setting we put on our categorical logicians' hats and prove meta-theorems
about provability statements and topos models. In these arguments we use classical
reasoning when necessary, including for Subsection~\ref{subsec:realisability} the Axiom of
Choice.

%\subsection{Basic definitions}
%\label{subsec:basic-definitions}

Let us recall some basic notions and terminology.
If $P$ is a poset, a \emph{chain} in $P$ is a subset $C \subseteq P$
such that for all $x, y \in C$, $x \leq y$ or $y \leq x$. The set of chains in $P$
is denoted by $\Ch{P}$.
A subset $D \subseteq P$ is \emph{directed} when every finite subset of~$D$,
including the empty set, has an upper bound in $D$; equivalently, if $D$ is
inhabited and every two elements in $D$ have a common upper bound in $D$.

A poset $P$ is \emph{chain-complete} if every chain in $P$ has a
supremum, and is \emph{directed-complete} if every directed subset of~$P$ has
a supremum. 
Any chain-complete poset is inhabited by the supremum of the empty chain, 
whereas a directed-complete poset may be empty.  However, any directed-complete 
poset with a bottom element is chain-complete: if $C$ is a chain, then 
$C \cup \set{\bot}$ is directed, and its supremum gives a supremum for $C$.
Since suprema are unique when they exist, a poset is chain-complete precisely
when it has a supremum operator $\sup : \Ch{P} \to P$.

An endofunction $f:P \to P$ is called \emph{progressive} (sometimes \emph{inflationary} or
\emph{increasing}) if $x \leq f(x)$ for every $x \in P$.
A point $x \in P$ is \emph{fixed} by $f$ if $f(x) = x$, \emph{pre-fixed} if $f(x) \leq x$,
and \emph{post-fixed} if $x \leq f(x)$.

The \emph{Bourbaki-Witt principle} is the statement
\begin{quote}
  \emph{``A progressive map on a chain-complete poset has a fixed
    point above every point.''}
\end{quote}

%%% Local Variables: 
%%% mode: latex
%%% TeX-master: "bw"
%%% End: 

\section{Bourbaki-Witt in the constructive setting}
\label{sec:constructive}

\subsection{Related fixed-point principles}
\label{subsec:related-fixed-point}

The Bourbaki-Witt principle is one of a family of fixed-point principles,
obtained by combining either progressive or monotone maps with either complete,
directed-complete, or chain-complete posets. Three of the six combinations can
be proved intuitionistically, as follows.

\begin{theorem}[Tarski~\cite{tarski55:_lattic_theor_fixpoin_theor_and_its_applic}]
  Any monotone map on a complete lattice has a fixed point above every post-fixed point.
\end{theorem}

\begin{proof}
  Let $f : P \to P$ be such a map and $x \in P$ a post-fixed point, i.e., $x \leq f(x)$.
  Consider the set $S = \set{y \in P \such \text{$x \leq y$ and $f(y) \leq y$}}$ of
  pre-fixed points above~$x$. The infimum $z = \inf S$ is a pre-fixed point because by
  monotonicity $f(z) \leq f(y) \leq y$ for all $y \in S$. But also $x \leq z$, so
  $z$ and $f(z)$ are in $S$, hence $z$ is a post-fixed point as well. Thus $z$ is a
  fixed point of $f$ above $x$, and indeed by construction the least such.
\end{proof}

\noindent
The usual formulation of Tarski's theorem states just that every monotone map has a fixed
point; here we reformulate it to make it more similar to the Bourbaki-Witt theorem, but
the two versions are equivalent.

\begin{theorem}[Pataraia~\cite{pataraia97:_const_proof_of_tarsk_fixed}]
  \label{theorem:pataraia}
  Any monotone map on a directed-complete poset has a fixed point above every
  post-fixed point.
\end{theorem}

\begin{proof}
  We summarise the proof as given by Dacar~\cite{Dacar:Join_Induction}. Given a monotone
  $f : P \to P$ on a directed-complete poset $P$, let $Q = \set{x \in P \such x \leq
    f(x)}$ be the subposet of post-fixed points. The set
  \begin{equation*}
    M = \set{ g : Q \to Q \such \text{$g$ is monotone and progressive}}
  \end{equation*}
  contains the restriction of $f$ to $Q$, is directed-complete under the pointwise
  ordering, and is itself directed: it contains the identity, and for any $g, h \in M$,
  the composite $g \circ h$ gives an upper bound of $g$ and $h$. Thus $M$ has a top
  element $t$, which must satisfy $g \circ t = t$ for all $g \in M$, hence $t(x)$ is a
  fixed point of~$f$ above~$x$ for any $x \in Q$.
\end{proof}

The third theorem which can be proved intuitionistically combines progressive maps and
complete lattices, but it is completely trivial as the top element is always a fixed point
of a progressive map. One might be tempted to save the theorem by proving that a
progressive map on a complete lattice has a \emph{least} fixed point, until one is shown a
counterexample.

The remaining three combinations claim existence of fixed points of a progressive map on a
chain-complete poset, a progressive map on a directed-complete poset, and a monotone map on
a chain-complete poset. The first of these is the Bourbaki-Witt principle, which we study in this
paper. Judging from Theorem~\ref{theorem:pataraia}, one might suspect that the second would have an
intuitionistic proof, but in fact Dacar~\cite{dacar08:_suprem_of_famil_of_closur_operat} has observed
that it is equivalent to the Bourbaki-Witt principle.

\begin{theorem}[Dacar]
  \label{theorem:bw-cc-iff-bw-dc}
  The following are intuitionistically equivalent:
  \begin{enumerate}
  \item
    \label{item:ccpo}
    Any progressive map on a chain-complete poset has a fixed
    point above every post-fixed point.
  \item
    \label{item:dcpo}
    Any progressive map on a directed-complete poset has a fixed
    point above every post-fixed point.
  \end{enumerate}
\end{theorem}

\begin{proof}
  The direction from chain-complete posets to directed-complete ones is straightforward: 
  if $P$ is directed-complete and $x$ is post-fixed for a progressive $f : P \to P$, then 
  $\set{y \in P \such x \leq y}$ is chain-complete and closed under $f$.
  
  To prove the converse, suppose the statement holds for directed-complete posets, and let
  $f : P \to P$ be a progressive map on a chain-complete poset $P$. The set $C$ of chains
  in~$P$, ordered by inclusion, is directed-complete. The map $F : C \to C$, defined by
  $F(A) = A \cup {f(\sup A)}$, is progressive, so has a fixed point~$B$
  above $\set{x}$. Now $f(\sup B) \in B$ and hence $f(\sup B) \leq \sup B$, showing
  that $\sup B$ is a fixed point of~$f$ above~$x$.
\end{proof}

The last combination is the Knaster-Tarski principle for chain-complete posets:
\begin{quote}
  \emph{``A monotone map on a chain-complete poset has a fixed point above every post-fixed point.''}
\end{quote}
Most of what we show for the Bourbaki-Witt principle in this paper holds almost
without alteration for the Knaster-Tarski principle, with one notable exception.  
As we saw in Theorem~\ref{theorem:pataraia}, the directed-complete version of the
Knaster-Tarski principle is intuitionistically provable, while the directed-complete
version of the Bourbaki-Witt theorem fails in general, as we will see in Section~\ref{sec:models}.

Finally, looking at the relationship \emph{between} the Knaster-Tarski and Bour\-baki-Witt
principles, we have:

\begin{proposition}
  \label{prop:bw-implies-tk}
  The Bourbaki-Witt principle implies the Knaster-Tarski principle.
\end{proposition}

\begin{proof}
  Let $f : P \to P$ be a monotone map on
  a chain-complete poset $P$, and suppose $x \leq f(x)$. Say that a chain
  $C \subseteq P$ is \emph{nice} if~$f$ is progressive on~$C$. Then
  the poset of nice chains under inclusion is chain-complete (indeed, directed-complete)
  and has a progressive map $s$, which sends $C$ to
  \begin{equation*}
    s(C) = C \cup \sup \set{f(y) \such y \in C}.
  \end{equation*}
  The Bourbaki-Witt principle gives a fixed-point $C$ of~$s$ above $\set{x}$.
  Then $\sup C$ is a fixed point of~$f$ above~$x$.
\end{proof}

\noindent We do not know whether this implication can be reversed!

  We summarize the intuitionistic provability of the six variants, and implications
between them, in the following diagram (where {\Checkmark}\kern-1.5pt\ stands for ``provable''):
\begin{center}
\small
\begin{tikzpicture}
  [x=2.5\baselineskip,y=2\baselineskip] % To alter scale, change these, not the individual co-ordinates.
\draw (1,-1.5) node {Progressive}
      (1,-2.5) node {Monotone}
      (3,-0.3) node {Complete}
      (7,-0.3) node [text width=2cm,text centered] {Chain-complete}
      (5,-0.3) node [text width=2cm,text centered] {Directed-complete};
\draw (3,-1.5) node {\Checkmark} % I don’t like this tick/cross pair much,
      (7,-1.5) node {\XSolidBrush} % but I couldn’t find a better one.
      (5,-1.5) node {\XSolidBrush} % If you have a better preferred pair, that’d be great.
      (3,-2.5) node {\Checkmark}
      (7,-2.5) node {\XSolidBrush}
      (5,-2.5) node {\Checkmark};
\draw (6,-1.5) node [fill=white] {$\Longleftrightarrow$};
\draw (7,-2) node [fill=white] {$\Downarrow$};  
% If we want to be more flexible with the implication arrows, then something like
% \draw [-implies,double distance=1.8pt] (7,-1.8) -- (7,-2.2);
% will give us the right style.
\end{tikzpicture}
\end{center}

%%% Local Variables: 
%%% mode: latex
%%% TeX-master: "bw"
%%% End: 

\subsection{Equivalent forms of Bourbaki-Witt}
\label{subsec:subtleties}
Bourbaki-Witt may be stated in several slightly different forms, all classically equivalent.
In fact, they turn out to be intuitionistically equivalent as well.

\begin{theorem}
  \label{theorem:equivalent-bw}
  The following are intuitionistically equivalent:
  \begin{enumerate}
  \item \label{lbl:bw} Any progressive map on a chain-complete poset has
    a fixed point above every point.
  \item \label{lbl:bw-inhabited} Any progressive map on a
    chain-complete poset has a fixed point.
  \item \label{lbl:bw-operator} Every chain-complete poset has a fixed-point
    operator for progressive maps.
  \end{enumerate}
\end{theorem}

\begin{proof}
  Let us first establish the equivalence of the first two statements. Every
  chain-complete poset has a least element, the supremum of the empty chain,
  above which one may seek fixed points. Conversely, a fixed-point of a
  progressive map $f : P \to P$ above $x \in P$ the same thing as a fixed-point
  of $f$ restricted to the chain-complete subposet $\up{x} = \set{y \in P \such
    x \leq y}$.

  The third statement clearly implies the second one. Conversely, suppose the
  second statement holds. Take any chain-complete poset~$P$ and let $\Prog{P}$
  be the set of progressive maps on~$P$. We can endow the exponential
  $P^{\Prog{P}}$ with a chain-complete partial order, defined by
  \begin{equation*}
    \tuple{x_f}{f} \leq \tuple{y_f}{f} \iff \xall{f}{\Prog{P}}{x_f \leq y_f},
  \end{equation*}
  where we write $\tuple{x_f}{f}$ for the element of $P^{\Prog{P}}$ that maps $f$ to
  $x_f$. The endomap $h : P^{\Prog{P}} \to P^{\Prog{P}}$,
  \begin{equation}
    \label{eq:apply-f-map}
    h(\tuple{x_f}{f}) = \tuple{f(x_f)}{f},
  \end{equation}
  is progressive, and so has a fixed point, which is exactly
  the desired fixed-point operator.
\end{proof}

Any of the the statements from Theorem~\ref{theorem:equivalent-bw} may be
interpreted in the internal language of a topos~$\E$. When we do so we refer to
them as the \emph{internal} Bourbaki-Witt principle. One may also consider
\emph{external} versions in which the universal quantifiers range externally
over progressive morphisms, rather than internally over the object of
progressive maps.
A morphism $f : P \to P$ is \emph{progressive} if it is so in the internal logic; equivalently, if $(\id{P}, f) : P \to P \times P$ factors through~$\leq$, viewed as a subobject of $P \times P$.

\begin{theorem}
  The \emph{internal} and \emph{external} Bourbaki-Witt theorems are
  equivalent in a topos~$\E$:
  \begin{enumerate}
  \item \label{lbl:bw-internal}
    Internal: for every chain-complete poset $P$ in~$\E$, the
    statement
    \begin{equation*}
      \xall{f}{P^P}{(\xall{x}{P}{x \leq f(x)}) \lthen \xsome{x}{P}{f(x) = x}}.
    \end{equation*}
    is valid in the internal logic of~$\E$.

  \item \label{lbl:bw-external}
    External: for every chain-complete poset $P$ in~$\E$ and
    every progressive morphism $f : P \to P$ the internal statement
    $\xsome{x}{P}{f(x) = x}$ is valid.
  \end{enumerate}
\end{theorem}

\begin{proof}
  The internal form obviously implies the external one. Conversely, suppose the
  external form holds, and consider any chain-complete poset $P$
  in~$\E$. As in the proof of Theorem~\ref{theorem:equivalent-bw}, we may
  construct in $\E$ the chain-complete poset $P^{\Prog{P}}$, and the canonical
  progressive morphism~$h$ thereon. By~\eqref{lbl:bw-external}, the statement
  $\xsome{z}{P^{\Prog{P}}}{h(z) = z}$ holds in~$\E$. We now conclude, just as in
  the proof of Theorem~\ref{theorem:equivalent-bw}, that there exists in the
  internal sense a fixed-point operator for~$P$, which implies the internal
  form.
\end{proof}

Similarly, various forms of the \emph{failure} of
Bourbaki-Witt turn out to be equivalent.  The failure of a universal statement is
generally weaker, intuitionistically, than the existence of a specific counterexample; 
and for the negation of the full, unbounded Bourbaki-Witt principle, this seems to be the case.
(Indeed, in topos logic, with no unbounded quantifiers, this negation cannot even be stated.)
However, as soon as the failure is in any way \emph{bounded}, one can construct a counterexample.

\begin{theorem}
  \label{theorem:not-bw-equivalent}
  The following are intuitionistically equivalent:
  \begin{enumerate}
  \item \label{lbl:bw-fail-exists-exists-not}
    There is a chain-complete poset and a progressive map on it
    which has no fixed points.
  \item \label{lbl:bw-fail-exists-not-forall}
    There is a chain-complete poset on which not every
    progressive map has a fixed point.
  \item \label{lbl:bw-fail-not-forall-forall}
    There is a set $\W$ of
    chain-complete posets such that not every progressive map on every poset in
    $\W$ has a fixed point.
 \end{enumerate}
\end{theorem}

\begin{proof}
  Clearly, the first statement implies the second one, which implies the third. To
  close the circle, suppose $\W$ is a set of chain-complete posets as in
  the third statement. Then the chain-complete poset $\prod_{P \in \W} P^{\Prog{P}}$
  carries a progressive endomap with no fixed point, sending $F$ to $(P, f) \mapsto f (F P
  f)$.
\end{proof}

We remark that the key ingredient in most proofs from this subsection was that any
product of chain-complete partial orders is again chain-complete.
Lemma~\ref{lemma:chain-completeness-drops} below may be seen as a strong generalisation of
this fact.

%%% Local Variables: 
%%% mode: latex
%%% TeX-master: "bw"
%%% End: 

%\input{tarski.tex} OBSOLETE
\subsection{A set of all trichotomous ordinals?}
\label{subsec:ordinals}

In the (futile) search for an intuitionistic proof of the Bourbaki-Witt
theorem it seems natural to consider the transfinite iteration of a
progressive map $f : P \to P$,
\begin{equation*}
  x \leq f(x) \leq f^2(x) \leq \cdots \leq f^\omega(x) \leq
  f^{\omega+1}(x) \leq \cdots
\end{equation*}
One feels that a fixed point will be reached, if only we can produce a sufficiently long
order to iterate along. In classical set theory this is possible, even without the axiom
of choice. For example, Lang~\cite{Lang:algebra} proves the Bourbaki-Witt theorem by
considering the least subset $C \subseteq P$ which contains~$x$, is closed under~$f$ and
under suprema of chains. He proves, classically but without choice, that $C$ is a chain,
from which it quickly follows that the supremum $\tbigwedge C$ is a fixed point of~$f$. In
fact, $C$ is (isomorphic to) an ordinal and is precisely large enough for the iteration
of~$f$ to stabilise after $C$-many steps.

Can fixed points always be found by transfinite iteration, as long as they exist? Is
failure of Bourbaki-Witt always due to a lack of existence of long enough ordinals? In
Subsection~\ref{subsec:not-enough-ordinals} below, we answer the first question
negatively: there is a topos in which the Bourbaki-Witt principle holds, but fixed points
cannot generally be reached by iteration along ordinals. In this section, we show that the
answer to the second question is positive: the Bourbaki-Witt principle fails precisely
when there is a set of all ordinals.

In the intuitionistic world the matter is complicated by the fact that the intuitionistic
theory of ordinals is not nearly so well behaved as the classical;
see~\cite{Taylor:intset} for an analysis of what can be done. Thus, before proceeding, we
need to pick a definition of ordinals.

Recall that a relation $<$ on $L$ is \emph{inductive} if it satisfies the induction
principle
\begin{equation*}
  (\xall{x}{L}{(\xuall{y < x}{\phi(y)}) \lthen \phi(x)})
  \implies
  \xall{x}{L}{\phi(x)},
\end{equation*}
for all predicates~$\phi$ on~$L$. In addition to the induction principle for predicates,
an inductive relation admits inductive definitions of maps. However, in our case,
attempting to iterate a progressive map, there is a complication. Given a progressive map
$f : P \to P$ on a chain-complete poset~$P$, we would like to define $\tilde{f} : L \to P$
inductively by
\begin{equation*}
  \tilde{f}(y) = \tbigwedge_{x < y} f(\tilde{f}(x)).
\end{equation*}
For this to be a valid definition we need to know that these suprema exist, so we must
ensure inductively that each $\set{f(\tilde{f}(x)) \such x < y}$ is a chain in~$P$. A
fairly strong notion of ordinals is needed:

\begin{definition}
  A \emph{trichotomous ordinal} $(L, {<})$, is a transitive inductive relation satisfying
  the law of trichotomy: for all $x, y \in L$, either $x < y$, $x = y$, or $y < x$.
\end{definition}

\noindent
One can now show:

\begin{lemma}
  If $L$ is a trichotomous ordinal, and $f$ is a progressive map on a chain-complete poset
  $P$, then we may define the iteration $\tilde{f} : L \to P$ of $f$ along $L$ as described above, by the
  equation $\tilde{f}(y) = \tbigwedge_{x < y} f(\tilde{f}(x))$.
\end{lemma}

\begin{proof}
  By induction on $y$, $\tilde{f}$ is monotone whenever it is defined; so
  $\set{\tilde{f}(x) \such x < y}$ is always a chain in~$P$, and thus $\tilde{f}$ is
  totally defined on $L$.
\end{proof}

A few more observations about trichotomous ordinals, similarly straightforward by
induction, will also be useful:
\begin{enumerate}
\item A inductive relation is asymmetric---that is, $(x < y) \lthen \lnot (y < x)$ for
  all $x,y$---and irreflexive.

\item Trichotomous ordinals are rigid: the only automorphism $L \to L$ is the identity.

\item The class of trichotomous ordinals forms a pre-order under the ``embeds as an
  initial segment'' relation, and is moreover chain-complete.

\item If $L$ is a trichotomous ordinal, then so is the strict order $L+1$ formed by
  adjoining a new top element above~$L$. This ordinal is called the \emph{successor} of
  $L$; the successor map on the class of trichotomous ordinals is progressive and has no
  fixed point. Note that unlike classically, the successor map may \emph{not} be monotone
  \cite{Taylor:intset}.
\end{enumerate}

We are now equipped to compare ordinal existence and Bourbaki-Witt as promised.

\begin{theorem}
  \label{theorem:ord-is-set}
  The following are (intuitionistically) equivalent:
  \begin{enumerate}
  \item \label{equivs:bw-counterexample} There is a progressive map on a chain-complete poset which has no fixed point.
  \item \label{equivs:ordinals-embed} There is a set into which every trichotomous ordinal injects.
  \item \label{equivs:ordwk-exists} There is a set $\Onwk$ of trichotomous ordinals such that every trichotomous ordinal is isomorphic to some ordinal in $\Onwk$.
  \item \label{equivs:ord-exists} There is a set $\On$ of trichotomous ordinals such that every trichotomous ordinal is isomorphic to a unique ordinal in $\On$.  (In topos-theoretic terms, $\On$ is a classifying object for trichotomous ordinals.) 
  \end{enumerate}
\end{theorem}

\begin{proof}
  We prove four implications: $(\ref{equivs:bw-counterexample}) \lthen
  (\ref{equivs:ordinals-embed}) \lthen (\ref{equivs:ordwk-exists}) \lthen
  (\ref{equivs:ord-exists}) \lthen (\ref{equivs:bw-counterexample})$.

  First, suppose $P$ is chain-complete and $f : P \to P$ is a progressive map without
  fixed points. For any trichotomous ordinal $L$, we can define the iteration $\tilde{f}$
  of~$f$ along $L$ as described above.
  But now, the map $\tilde{f}$ is injective: if $\tilde{f}(x) = \tilde{f}(y)$, then $x <
  y$ cannot hold because that would give us a fixed point of~$f$:
  \begin{equation*}
    f(\tilde{f}(x)) \leq \tbigwedge_{x < y} f(\tilde{f}(x)) =
    \tilde{f}(y) = \tilde{f}(x) \leq f(\tilde{f}(x)).
  \end{equation*}
  The case $y < x$ is similarly impossible, so $x = y$. Thus every trichotomous ordinal
  embeds in $P$, as required.
  
  Second, if every ordinal injects into a set $A$, then take
  \begin{equation*}
    \Onwk = \set{ (L, {<}) \in \pow{A} \times \pow{A \times A} \such
      \text{$(L,{<})$ is a trichotomous ordinal}}.
  \end{equation*}

  In the third implication we avoid the axiom of choice by using an idea familiar from the
  construction of moduli spaces in geometry: if we can weakly classify a class of objects
  and they have no non-trivial automorphisms, then we can classify them.
  Take the quotient set $\Onwk/{\iso}$ of equivalence classes of ordinals up to
  isomorphism. Now for any equivalence class $C \in \Onwk/{\iso}$, we can define a
  canonical representative as follows. Take the coproduct $S_{C} = \coprod_{L \in C} L$,
  and for $L, L' \in C$, $x \in L$, $y \in L'$, set $x \sim y$ if the unique isomorphism
  $L \iso L'$ sends $x$ to $y$. Then $R_{C} = S_{C}/{\sim}$ has a natural bijection to
  each $L \in C$, commuting with the isomorphisms between these; so with the ordering
  transferred along any of these bijections, $R_C$ is a trichotomous ordinal, and a
  representative for $C$. Thus $\On = \set{R_C \such C \in \Onwk/{\iso}}$ is as desired.

  The last implication is easy because the set $\On$ of trichotomous ordinals, if it
  exists, is a chain-complete poset under the initial-segment preorder; and the successor
  map on~$\On$ is progressive and has no fixed points.
\end{proof}

%%% Local Variables: 
%%% mode: latex
%%% TeX-master: "bw"
%%% End: 

\section{Topos models} \label{sec:models}

% I'm not sure if I'm happy with the subsection titles "B-W fails in...", "B-W
% holds in..." --- it ends up repeating "Bourbaki-Witt" so many times! "Failure
% in realisability toposes" etc. would be nicer for the negative ones, but I
% can't think of a similarly suitable nounification for "B-W holds in sheaf
% toposes"? --- PLL 2009

% Hmm, on reflection I quite like this naming scheme. --- PLL 2011

\subsection{Bourbaki-Witt fails in realisability toposes}
\label{subsec:realisability}

The Bourbaki-Witt principle fails in the effective topos $\Eff$, as was shown by the
first author~\cite{Bauer:On_the_Failure}. We indicate how the proof can be
adapted easily to work in any realisability topos. For background on
realisability see~\cite{oosten08:_realiz}.

Let $A$ be a partial combinatory algebra and $\RT{A}$ the realisability topos over it. The
category of sets~$\Set$ is equivalent to the category of sheaves in $\RT{A}$ for the
$\lnot\lnot$-coverage. The inverse image part of the inclusion $\RT{A} \to \Set$ is the
global points functor $\Gamma : \RT{A} \to \Set$, and we denote the direct image by
$\nabla : \Set \to \RT{A}$.

Let $\kappa$ be the cardinality of $A$, where we work classically in $\Set$. The successor
$\kappa^{+}$ is a regular cardinal, which we view as an ordinal. The successor map $s :
\kappa^{+} \to \kappa^{+}$ is progressive and monotone but has no fixed points. This is no
suprise as $\kappa^{+}$ is not chain-complete, although it has suprema of chains whose
cardinality does not exceed $\kappa$. But the poset $\nabla \kappa^{+}$ \emph{is}
chain-complete in $\RT{A}$ because every chain in $\RT{A}$ has at most $\kappa$ elements
(to see what exactly this means in the internal language of $\RT{A}$
consult~\cite{Bauer:On_the_Failure}), therefore the successor map $\nabla s : \nabla
\kappa^{+} \to \nabla \kappa^{+}$ provides a counterexample to both the Bourbaki-Witt and
the Knaster-Tarski principle.

In the effective topos $\Eff$ the object $\On'$ from
Theorem~\ref{theorem:ord-is-set} has a familiar description. It is none other
than Kleene's universal system of notations $O$ for recursive ordinals,
see~\cite[11.7]{RogersH:recfun}.

%%% Local Variables: 
%%% mode: latex
%%% TeX-master: "bw"
%%% End: 

\subsection{Bourbaki-Witt does not hold in the free topos}

Recall \cite{lambek-and-scott} that there is an elementary topos $\Efree$, ``the free topos'', constructed from the syntax of intuitionistic higher-order logic (IHOL), and pseudo-initial in the category of elementary toposes and logical morphisms. Objects in $\Efree$ are thus exactly such objects as are definable in IHOL, and have exactly such properties as are provable.

Does Bourbaki-Witt hold in the free topos? It cannot \emph{fail}, since the canonical
logical morphism $\Efree \to \Set$ would preserve any failure. But it might not hold either:
there could be some poset defined in IHOL, provably chain-complete, with a definable and
provably progressive map, for which the existence of a fixed point is not provable. To
show this unprovability for some particular $P$ and $f$, it suffices to give a topos $\E$
in which the interpretation of $f$ has no fixed point. Happily, with just a little work,
the poset $\nabla \omega_1$ in $\Eff$ (an instance of the construction of Subsection~\ref{subsec:realisability}), and its successor map, can
be exhibited as such an interpretation.

\begin{theorem}
  The Bourbaki-Witt principle does not hold in $\Efree$.
\end{theorem}

\begin{proof}
  As we saw above, $\nabla$ embeds $\Set$ as sheaves for the $\lnot \lnot$ topology on
  $\Eff$. In $\Set$, $\omega_1$ is definable as a subquotient of $2^\N$: the set of all
  subsets of $\N \times \N$ describing well-orderings of $\N$, modulo isomorphism of the
  resulting well-orders. Thus, interpreting this definition in the Kripke-Joyal semantics
  for $\lnot \lnot$ in $\Eff$, $\nabla(\omega_1)$ is definable as the $\lnot
  \lnot$-sheafification of a certain quotient of a certain subobject of ${\Omega_{\lnot
      \lnot}}^{\N \times \N}$; similarly, its order and the successor map are definable,
  so we have a poset $\omega_1^{\lnot \lnot}$ in $\Efree$, together with a progressive
  endomap $s$, which are interpreted as $\nabla(\omega_1)$ and its successor map in
  $\Eff$.

  Unfortunately, $\omega_1^{\lnot \lnot}$ cannot be chain-complete in $\Efree$, since in
  $\Set$ it is interpreted as $\omega_1$. We can remedy this, however, using an
  exponenential by a truth-value. Let $t$ denote the set $\{ * \in 1 \such \omega_1^{\lnot
    \lnot}\ \textrm{is chain-complete}\}$, and set
  $$P := (\omega_1^{\lnot \lnot})^t = \prod_{u \in t} \omega_1^{\lnot \lnot}.$$
  This now has a natural chain-complete ordering, since the second description exhibits it
  as a dependent product of chain-complete posets: $\omega_1^{\lnot \lnot}$ is not in
  general chain-complete, but given any $u \in t$, it certainly is! Similarly, the endomap
  $f = s^t$ is progressive. But in $\Eff$, the truth-value in question is $1$, so $P$ is
  interpreted as $\nabla(\omega_1)^1 \iso \nabla(\omega_1)$, and $f$ as successor. Thus
  the existence of a fixed point of $f$ is not provable, so we have a non-example of
  Bourbaki-Witt in $\Efree$.
\end{proof}

Taking exponentials by truth-values in this fashion may be seen as an intuitionistic
implementation of the classical construction ``if $P$ is chain-complete then $P$, else
$1$''.

%%% Local Variables: 
%%% mode: latex
%%% TeX-master: "bw"
%%% End: 

\subsection{Bourbaki-Witt holds in cocomplete toposes}

We have seen that the Bourbaki-Witt and Tarski conditions are not in general
constructively valid. However, they hold in an important class of models thanks to the
following transfer principle.

\begin{theorem}
  \label{theorem:lift-along-gm}%
  If $\E \to \F$ is a geometric morphism and $\F$ satisfies the Bourbaki-Witt
  principle, then so does $\E$.
\end{theorem}

In particular, any cocomplete topos $\E$ has a geometric morphism $(\Gamma, \Delta) : \E \to \Set$, where
$\Gamma(A) = \E(1,A)$ is the global-points functor and $\Delta(X) = \coprod_{X} 1$ takes a
set $X$ to the $X$-fold coproduct of~$1$'s. By applying the theorem to this case, we see
that the Bourbaki-Witt principle holds in cocomplete toposes:

\begin{corollary}
  \label{corollary:bw-in-grothendieck}%
  Any cocomplete topos, in particular any sheaf topos, satisfies the Bourbaki-Witt
  principle.
\end{corollary}

\noindent
Since this is our guiding example, we will write the geometric morphism as
$(\Gamma,\Delta)$ in general, for the comforting familiarity it provides.
To prove the theorem one requires a main lemma:

\begin{lemma}
  \label{lemma:chain-completeness-drops}%
  If $(\Delta,\Gamma) : \E \to \F$ is a geometric morphism and $P$ is chain-complete in
  $\E$, then $\Gamma(P)$ is chain-complete in $\F$.
\end{lemma}

\begin{proof}
  We wish to construct a supremum map $\tbigvee_{\Gamma P} : \Ch{\Gamma P} \to \Gamma P$.
  % i.e.\ such that $\E$ validates ``for all $\tbigvee_{\Gamma P} : \Ch(\Gamma P)$, $y:P$, $\tbigvee_{\Gamma P}(c) \leq y \liff \all[x][c][x \leq y]$''. 
  Consider the universal chain in $\Gamma P$, i.e.\ the $\Ch{\Gamma P}$-indexed subset of
  $\Gamma P$
  $$C = \{ (x, c) \such c \in x \} \into \Ch{\Gamma P} \times \Gamma P.$$
  $\Delta C$ is now a $\Delta(\Ch{\Gamma P})$-indexed subset of $\Delta \Gamma P$, and
  indeed is a chain, since $\Delta$ preserves $\lor$; so its image $\widehat{C}$ under
  $\epsilon_P : \Delta \Gamma P \to P$ (the co-unit of the geometric morphism) is a
  $\Delta(\Ch{\Gamma P})$-indexed chain in $P$. Thus there is a map $s: \Delta(\Ch{\Gamma
    P}) \to P$ giving suprema for $\widehat{C}$, and hence for $\Delta(C)$.

  Its transpose $\check{s}: \Ch{\Gamma P} \to \Gamma P$ is our candidate for
  $\tbigvee_{\Gamma P}$. We just need to show that $\E$ validates ``for all $c : \Ch{\Gamma
    P}$ and $x:P$, $\check{s}(c) \leq x \liff \all{y}{c}{y \leq x}$'', or in other words,
  that for any $(c,x): A \to \Ch{\Gamma P} \times \Gamma P$, the map
  $$(\check{s} \circ c, x): A \to \Gamma P \times \Gamma P$$
  factors through $\Gamma (\leq)$ if and only if the map 
  $$m: C \times_{\Ch{\Gamma P}} A = \{ (y,a) \such y \in c(a) \} \to \Gamma P \times \Gamma P$$
  sending $(y,a)$ to $(y,x(a))$ factors through $\Gamma (\leq)$.

  But by the universal property of the adjunction, $(\check{s} \circ c, x)$ factors
  through $\Gamma (\leq)$ if and only if its transpose
  $$\widehat{(\check{s} \circ c, x)} = ((s \circ \Delta(c)), \hat{x}) : \Delta(A) \to P \times P$$
  factors through $\leq$. Since $s$ gives suprema for $\Delta(C)$, this in turn happens if
  and only if the map
  $$\hat{m} : \Delta(C) \times_{\Delta \Ch{\Gamma P}} \Delta(A) \to P \times P$$
  sending $(y,a)$ to $(y,\hat{x}(a))$ factors through $\leq$. But $\hat{m}$ is just the
  transpose of $m$, and so $\hat{m}$ factors through $\leq$ exactly if $m$ factors through
  $\Gamma(\leq)$. Thus $\check{s}$ gives suprema for chains in $\Gamma P$, as desired.
\end{proof}

\begin{proofof}{Theorem \ref{theorem:lift-along-gm}} Suppose now that $P$ is a chain-complete poset in $\E$, $f:P \to P$ is progressive, and $\F$ satisfies the Bourbaki-Witt principle.

$\Gamma(P)$ is chain-complete, by Lemma \ref{lemma:chain-completeness-drops}, and $\Gamma(f)$ is progressive, so $\F$ validates ``$\Gamma(f)$ has some fixed point in $\Gamma(P)$''.  Being a statement of geometric logic, this is preserved by $\Delta$, so $\E$ validates ``$\Delta (\Gamma(f))$ has some fixed point in $\Delta(\Gamma(P))$''.

But now $\epsilon_P \circ \Delta(\Gamma(f)) = f \circ \epsilon_P$ (by naturality), so if $x \in \Delta(\Gamma(P))$ is any fixed point of $\Delta(\Gamma(f))$, then $\epsilon_P(x) \in P$ is a fixed point of $f$.  So $\E$ validates ``$f$ has some fixed point in $P$'', as desired.
\end{proofof}

The only point in this section at which classical logic is required is for Corollary~\ref{corollary:bw-in-grothendieck}, to know that the Bourbaki-Witt theorem holds in $\Set$.

%%% Local Variables: 
%%% mode: latex
%%% TeX-master: "bw"
%%% End: 

\subsection{Bourbaki-Witt does not imply ordinal existence} \label{subsec:not-enough-ordinals}

In Subsection \ref{subsec:ordinals} above, we asked: if Bourbaki-Witt holds, can any fixed point be computed by some long enough ordinal iteration?  Here, we present a counterexample: a topos in which Bourbaki-Witt holds, but there are not enough ordinals to compute fixed points.

The rough idea is as follows: we first consider ordinals and posets in the presheaf topos $\Sets^\arr$, where an ordinal turns out to be a pair of ordinary ordinals with a strictly monotone map between them, written as $[L_1 \to L_0]$.  Since in any ordinal, ${<}$ implies ${\neq}$, the length of the first component $L_1$ is bounded by the length of its second component $L_0$.  By contrast, looking at chain-complete posets $[P_1 \to P_0]$ with progressive maps, the length of iteration required to find fixed points can be made arbitrarily large by blowing up just $P_1$, while holding $P_0$ fixed.

So in any assignment $(P,f) \mapsto L$ providing ordinals to compute fixed points, $L_0$ must depend on $P_1$, not only on $P_0$.  But in any purely logical (i.e.~IHOL) construction, $L_0$ would depend only on $P_0$, by the construction of the logical structure in $\Sets^\arr$.
So although $\Sets^\arr$ has enough ordinals to compute fixed points, this fact cannot be realised by any purely logical construction.

Thus in $\EBWPf$, the free topos satisfying the Bourbaki-Witt principle and with a distinguished chain-complete poset and monotone map, there cannot be any ordinal computing the fixed point of $f$, since this would give a logical construction of such ordinals in any other topos, which we have seen is not possible.

We now formalise this argument, first setting up some terminology for the eventual goal.

\begin{definition}Say that a topos $\E$ satisfying the Bourbaki-Witt principle \emph{has enough ordinals} if for any chain-complete poset $P$ in $\E$ with a progressive map $f$, there is some object $B$, inhabited in the internal sense (i.e.~$B \to 1$ is epi), and some $B$-indexed family of ordinals $\la L_b \such b \in B \ra$, such that $\E$ validates ``for each $b \in B$, the iteration $f^{L_b}$ of $f$ along $L_b$ has as its supremum a fixed point of $f$''.  (We say that the ordinals $L_b$ \emph{compute fixed points} for $f$.)
\end{definition}

\begin{definition}Let $\LL_\BW[P,f]$ be the theory in IHOL given by adding to pure type theory an axiom schema asserting that the Bourbaki-Witt principle holds, together with a new type $P_\univ$, constants ${\leq}$ and $f_\univ$, and axioms asserting that $(P_\univ,\leq)$ is a chain-complete poset and $f_\univ$ a progressive map thereon.  Let $\EBWPf$ be the syntactic topos of this type theory \cite[II.11--16]{lambek-and-scott}.
\end{definition}

The universal property of $\EBWPf$ tells us that given any topos $\F$ satisfying Bourbaki-Witt and a progressive endomap $f$ on a chain-complete poset $P$ therein, there is a logical functor $\EBWPf \to \F$, unique up to canonical natural isomorphism, sending $P_\univ$ and $f_\univ$ to $P$ and $f$ respectively.\footnote{Contrary to what one might at first expect, these will not be the only logical functors out of $\EBWPf$; the axiom schema only forces Bourbaki-Witt to hold for posets in the image of the functor, not in the whole target topos.}

The goal of this section is now:

\begin{theorem} \label{thm:not-enough-ordinals}
  The topos $\EBWPf$ does not have enough ordinals.
  In particular, there is no inhabited family of ordinals $L \to B \epito 1$ that computes fixed points for $f_\univ$.
\end{theorem}

As indicated above, we begin by investigating ordinals and partial orders in $\Sets^\arr$.
We will write objects of $\Sets^\arr$ as $X = [X_1 \to X_0]$.
Since the functors $\ev_1, \ev_0 \colon \Sets^\arr \to \Sets$ preserve finite limits, we may similarly write any (stric) partial order $P$ in $\Sets^\arr$ as a map $[P_1 \to P0]$ of (strict) partial orders in $\Sets$.

The functor $\ev_0$ is moreover logical; so if $L$ is an ordinal in $\Sets^\arr$, then $L_0$ is an ordinal in $\Sets$.
More generally, any slice functor $\ev_0/X \colon \Sets^\arr / X \to \Sets / X_0$ is logical; so if $L \to B$ is a family of ordinals in $\Sets^\arr$, then $L_0 \to B_0$ is a family of ordinals $\la (L_0)_b\ |\ b \in B_0 \ra$.

\begin{lemma} \label{lemma:poset-blowup}
  For any ordinal $\alpha$ in $\Sets$, let $P_\alpha$ be the poset $[\alpha + 1 \to 1]$ in $\Sets^\arr$, and $f_\alpha$ the progressive endomap of $P_\alpha$ that acts as successor on $\alpha$, and as the identity elsewhere.  Then:
  \begin{enumerate}
    \item \label{p_alpha_cc} $P_\alpha$ is chain-complete in $\Sets^\arr$; and
    \item \label{p_needs_long} if $L \to B$ is any inhabited family of ordinals computing fixed points for $f_\alpha$, then $\sup_{b \in B_0} (L_0)_b > \alpha$.
    In other words, with $P_\alpha$, we have succeeded in blowing up the required length of $L_0$ to $\alpha + 1$, while holding $P_0$ constant at 1.
  \end{enumerate}
\end{lemma}

\begin{proof}
Chain-completeness follows immediately from Lemma~\ref{lemma:chain-completeness-drops}, since $\alpha + 1$ is chain-complete in $\Sets$, and the functor sending a set $X$ to $[X \to 1]$ is the forward image of a geometric morphism, with inverse image $\ev_1$.

Explicitly, the object of chains in $P_\alpha$ is given by
\[ \Ch{P_\alpha} = \pow{P_\alpha} \iso [ \set{S,T \such S \subseteq \alpha + 1,\, T \subseteq 1,\, \im(S) \subseteq T} \to \pow{1} \]
and the supremum map $\sup \colon \Ch{P_\alpha} \to P_\alpha$ is given by
\[
  \tsup_1(S,T) = \sup S \in \alpha + 1 
  \qquad
  \tsup_0(T) = * \in 1.
  \]
With this in hand, suppose $L \to B$ is some inhabited family of ordinals computing fixed points for $f_\alpha$, and let $\tilde{f} \colon L \to P_\alpha$ denote the iteration of $f_\alpha$ along $L$.

For any ordinals $\alpha,\beta$ in $\Sets$, there is a canonical map of partial orders $\alpha \to \beta + \set{\top}$, the initial-segment embedding if $\alpha \leq \beta$, or stabilising at the top if $\alpha > \beta$.  (One may regard this as a truncated rank function.)  Viewing $L_0$ as the disjoint union of the ordinals $\la (L_0)_b \such b \in B_0 \ra$, we obtain a notion of $(\alpha + \set{\top})$-valued rank for elements of $L_0$, and hence of $L_1$:
\[ \rk \colon L_1 \to L_0 \to \alpha + \set{\top}. \]
This is very nearly strictly monotone: if $x < y$, then either $\rk(x) < \rk(y)$ or $\rk(x) = \rk(y) = \top$.

Now, we see that for every $x \in L_1$, $\tilde{f}_1(x) \leq \rk(x)$.  If $\rk(x) = \top$, this is trivial; otherwise, we work by induction on $\rk(x)$:
\begin{align*}
  \tilde{f}_1(x) &= \tsup_1 ( \set{ \tilde{f}_1(y) \such y < x } , \set{ \tilde{f}_0(z) \such z < x\restrict{0}} ) \\
                 &= \sup \set{ \tilde{f}_1(y) \such y < x } \\
                 &\leq \sup \set{ \rk y \such y < x } \quad \text{(by induction)} \\
                 &\leq \rk x.
\end{align*}

Now, by hypothesis, $\Sets^\arr$ validates ``for each $b \in B$, $\sup \set{ \tilde{f}(i) \such i \in L_b }$ is a fixed point of $f_\alpha$''.  Since $B$ is inhabited, there is some $b \in B_1$; so calculating as above, we see that 
\begin{align*}
\top &= \sup \set{\tilde{f}_1(x) \such x \in (L_1)_b} \\
     &\leq \sup \set{ \rk x \such x \in (L_1)_b } \\
     &\leq \sup \set{ \rk i \such i \in (L_0)_{b \restrict{0}}},
\end{align*}
whence the ordinal $(L_0)_{b \restrict{0}}$ must be at least $\alpha$, as desired.
\end{proof}

Finally, let us prove Theorem~\ref{thm:not-enough-ordinals}.
Let $L \to B$ be any inhabited family of ordinals in $\EBWPf$; we wish to show that $L$ does not compute fixed points for $f$.

Coonsider the logical morphism $F_1 \colon \EBWPf \to \Sets$ sending $(P,f)$ to the terminal poset $1$ and its unique endomap.  This sends $L$ to some inhabited family of ordinals $\la \lambda_b \such b \in F_1(B) \ra$; let $\alpha$ be an ordinal greater than the supremum of this family.

Now consider the logical morphism $F_{P_\alpha} \colon \EBWPf \to \Sets^\arr$, sending $(P,f)$ to $(P_\alpha,f_\alpha)$.  Since $\ev_0$ is a logical morphism and $(P_\alpha)_0 = 1$, the universal property of $\EBWPf$ enforces that $\ev_0 \circ F_{P_\alpha} = F_1$.  So, in particular, $F_{P_\alpha}(L)_0$ is again the family of ordinals $\la \lambda_b \such b \in F_1(B) \ra$, with supremum less than $\alpha$.

Thus, by Lemma \ref{lemma:poset-blowup}, the family of ordinals $F_{P_\alpha}(L)$ cannot compute fixed-points for $f_\alpha$ in $\Sets^\arr$.  So, since $F_{P_\alpha}$ is logical, $L$ cannot compute fixed points for $f$ in $\EBWPf$.  But $L$ was arbitrary; so no inhabited family of ordinals can suffice, and $\EBWPf$ does not have enough ordinals.

A word of caution is necessary here, however: all these negations have been \emph{external}, so for all we know it could still be the case that $\EBWPf$ validates the double-negated version ``$L$ does not fail to compute fixed points for $f$'', for some $L \to B$. 

%%% Local Variables: 
%%% mode: latex
%%% TeX-master: "bw"
%%% End: 

%%%%%%%%%%%%%%%%%%%%%%%%%%%%%%%%%%%%%%%%%%%%%%%%%%
% BIBLIOGRAPHY

\bibliographystyle{plain}
\bibliography{bw}

\end{document}